\documentclass[12pt]{amsart}

\usepackage{amsmath} \usepackage{amssymb} \usepackage{amsthm}
\usepackage{amscd}

\def\a{{\mathfrak{a}}} \def\b{{\mathfrak{b}}}
\def\F{{\mathbb{F}}} \def\J{{\mathcal{J}}} \def\m{{\mathfrak{m}}}\def\n{{\mathfrak{n}}} \def\Z{{\mathbb{Z}}}\def\N{{\mathbb{N}}}

\def\Q{{\mathbb{Q}}} \def\R{{\mathbb{R}}} 
  
  \def\Ann{{\mathrm{Ann}}}
 
\def\ord{{\mathrm{ord}}}

\def\Tor{\operatorname{Tor}}

\DeclareMathOperator{\pt}{c}

\theoremstyle{plain}
\newtheorem{thm}{Theorem}[section]
\newtheorem{cor}[thm]{Corollary}
\newtheorem{prop}[thm]{Proposition}
\newtheorem{conj}[thm]{Conjecture}

\newtheorem{lem}[thm]{Lemma}
\theoremstyle{definition}
\newtheorem{defn}[thm]{Definition}
\newtheorem{eg}[thm]{Example}
\theoremstyle{remark}
\newtheorem{rem}[thm]{Remark}
\newtheorem{ques}[thm]{Question}

\newtheorem*{cl}{Claim}

\newtheorem*{acknowledgement}{Acknowledgment}

\newcommand{\llbracket}{[\negthinspace[}
\newcommand{\rrbracket}{]\negthinspace]}

\begin{document}

\title{F-thresholds, tight closure, integral closure, and multiplicity bounds}
\author[Huneke]{Craig Huneke}
\address{Department of Mathematics, University of Kansas,
Lawrence, KS 66045-7523, USA} \email{huneke@math.ku.edu}

\author[Mustata]{Mircea Musta\c{t}\u{a}}
\address{Department of Mathematics, University of
Michigan\\ Ann Arbor, MI 48109, USA} \email{mmustata@umich.edu}

\author[Takagi]{Shunsuke Takagi}
\address{Department of Mathematics, Kyushu University, 6-10-1
Hakozaki, Higashi-ku, Fukuoka-city 812-8581, Japan}
\email{stakagi@math.kyushu-u.ac.jp}

\author[Watanabe]{Kei-ichi Watanabe}
\address{Department of Mathematics, College of Humanities and Sciences,
Nihon University, Setagaya-Ku, Tokyo 156-0045, Japan}
\email{watanabe@math.chs.nihon-u.ac.jp}


\baselineskip 16pt \footskip = 32pt

\begin{abstract}
The F-threshold $\pt^J(\a)$ of an ideal $\a$ with respect to the
ideal $J$ is a positive characteristic invariant obtained by
comparing the powers of $\a$ with the Frobenius powers of $J$. We
show that under mild assumptions, we can detect the containment in
the integral closure or the tight closure of a parameter ideal using
F-thresholds. We formulate a conjecture bounding $\pt^J(\a)$ in
terms of the multiplicities $e(\a)$ and $e(J)$, when $\a$ and $J$
are zero-dimensional ideals, and $J$ is generated by a system of
parameters. We prove the conjecture when $J$ is a monomial ideal in
a polynomial ring, and also when $\a$ and $J$ are generated by
homogeneous systems of parameters in a Cohen-Macaulay graded
$k$-algebra.
\end{abstract}

\maketitle \markboth{C.~Huneke, M.~Musta\c{t}\u{a}, S.~Takagi, and
K.-i.~Watanabe}{F-thresholds, tight closure, integral closure, and
multiplicity bounds}

\section*{Introduction}

Let $R$ be a Noetherian ring of positive characteristic $p$. For
every ideal $\a$ in $R$, and for every ideal $J$ whose radical
contains $\a$, one can define asymptotic invariants that measure the
containment of the powers of $\a$ in the Frobenius powers of $J$.
These invariants were introduced in the case of a regular local
F-finite ring in \cite{MTW}, where it was shown that they coincide
with the jumping exponents for the generalized test ideals of Hara
and Yoshida \cite{HY}. In this paper we work in a general setting,
and show that the F-thresholds still capture interesting and subtle
information. In particular, we relate them to tight closure and
integral closure, and to multiplicities.

If $\a$ and $J$ are as above, we define for every positive integer
$e$
$$\nu^J_{\a}(p^e):=\max\{r\mid\a^r\not\subseteq J^{[p^e]}\},$$
where $J^{[q]}$ is the ideal generated by the $p^e$-powers of the
elements of $J$. We put
$$\pt^J_+(\a):=\limsup_{e\to\infty}\frac{\nu^J_{\a}(p^e)}{p^e},\,\,\pt^J_{-}(\a):=
\liminf_{e\to\infty}\frac{\nu^J_{\a}(p^e)}{p^e},$$ and if these two
limits coincide, we denote their common value by $\pt^J(\a)$, and
call it the \emph{F-threshold of $\a$ with respect to $J$}.

Our first application of this notion is to the description of the
tight closure and of the integral closure for parameter ideals.
Suppose that $(R,\m)$ is a $d$-dimensional Noetherian local ring of
positive characteristic, and that $J$ is an ideal in $R$ generated
by a full system of parameters. We show that under mild conditions,
for every ideal $I\supseteq J$, we have $I\subseteq J^*$ if and only
if $\pt^I_+(J)=d$ (and in this case $\pt^I_{-}(J)=d$, too). We
similarly show that under suitable mild hypotheses, if $I\supseteq
J$, then $I\subseteq \overline{J}$ if and only if $\pt^J_+(I)=d$.
For the precise statements, see Corollary~\ref{cor_tight} and
Theorem~\ref{integral} below.

As we have mentioned, if $R$ is regular and F-finite, then it was
shown in \cite{MTW} that the F-thresholds of an ideal $\a$ coincide
with the jumping exponents for the generalized test ideals of
\cite{HY}. In order to recover such a result in a more general
setting, we develop a notion of F-threshold for the ideal $\a$
corresponding to a submodule $N$ of a module $M$, such that
$\a^nN=0$ for some $n$. We then show that under suitable hypotheses
on a local ring $R$, one can again recover the jumping exponents for
the generalized test ideals of an ideal $\a$ in $R$ from the
F-thresholds of $\a$ with respect to pairs $(E,N)$, where $N$ is a
submodule of the injective hull $E$ of the residue field (see
Corollary~\ref{cor-F-thr}).

We study the connection between F-thresholds and multiplicity, and
formulate the following conjecture: if $(R,\m)$ is a $d$-dimensional
Noetherian local ring of characteristic $p>0$, $\a$ and $J$ are
$\m$-primary ideals in $R$, with $J$ generated by a system of
parameters, then
$$e(\a)\geq \frac{d^d}{\pt^J_{-}(\a)^d}e(J).$$
The case $J=\m$ (when $R$ is in fact regular) was proved in
\cite{TW}. We mention that in this case $\pt^{\m}(\a)$ is related
via reduction mod $p$ to a fundamental invariant in birational
geometry, the log canonical threshold ${\rm lct}(\a)$ (see
\emph{loc. cit.} for the precise relation between these two
invariants). The corresponding inequality between the multiplicity
and the log canonical threshold of $\a$ was proved in \cite{dFEM},
and plays a key role in proving that for small values of $n$, no
smooth hypersurface of degree $n$ in ${\mathbb P}^n$ is rational
(see \cite{Corti} and \cite{dFEM2}).

We prove our conjecture when both $\a$ and $J$ are generated by
homogeneous systems of parameters in a graded Cohen-Macaulay
$k$-algebra (cf. Corollary~\ref{cor_homogeneous}). Moreover, we
prove it also when $R$ is regular and
$J=(x_1^{a_1},\ldots,x_n^{a_n})$, for a regular system of parameters
$x_1,\ldots,x_n$. The proof of this latter case follows the ideas in
\cite{TW} and \cite{dFEM}, reducing to the case of a monomial ideal
$\a$, and then using the explicit interpretation of the invariants
involved in terms of the Newton polyhedron of $\a$.

 On the other hand, the proof of the homogeneous
case is based on new ideas that we expect to be useful also in
attacking the general case of the conjecture. In fact, we prove the
following stronger statement. Suppose that $\a$ and $J$ are ideals
generated by homogeneous systems of parameters in a $d$-dimensional
graded Cohen-Macaulay $k$-algebra, where $k$ is a field of
\emph{arbitrary} characteristic. If $\a^N\subseteq J$ for some $N$,
then
$$e(\a)\geq\left(\frac{d}{d+N-1}\right)^d e(J).$$

The paper is structured as follows. In the first section we recall
some basic notions of tight closure theory, and review the
definition of generalized test ideals from \cite{HY}. In \S 2 we
introduce the F-thresholds and discuss some basic properties. The
third section is devoted to the connections with tight closure and
integral closure. We introduce the F-thresholds with respect to
pairs of modules in \S 4, and relate them to the jumping exponents
for the generalized test ideals. In the last section we discuss
inequalities involving F-thresholds and multiplicities. In
particular, we state here our conjecture and prove the
above-mentioned special cases.

\section{Preliminaries}
In this section we review some definitions and notation that will be
used throughout the paper. All rings are Noetherian commutative
rings with unity. For a ring $R$, we denote by $R^{\circ}$ the set
of elements of $R$ that are not contained in any minimal prime
ideal. Elements $x_1,\ldots,x_r$ in $R$ are called \emph{parameters}
if they generate an ideal of height $r$. The integral closure of an
ideal $\a$ is denoted by $\overline{\a}$. The order of a nonzero
element $f$ in a Noetherian local ring $(R,\m)$ is the largest $r$
such that $f\in\m^r$. For a real number $u$, we denote by $\lfloor
u\rfloor$ the largest integer $\leq u$, and by $\lceil u \rceil$ the
smallest integer $\geq u$.

Let $R$ be a ring of characteristic $p > 0$, and let $F\colon R \to
R$ denote the Frobenius map which sends $x \in R$ to $x^p \in R$.
The ring $R$ viewed as an $R$-module via the $e$-times iterated
Frobenius map $F^e \colon R \to R$ is denoted by ${}^e\! R$. We say
that $R$ is {\it F-finite} if ${}^1 \! R$ is a finitely generated
$R$-module. We also say that $R$ is {\it F-pure} if the Frobenius
map is pure, that is, $F_M=1_M \otimes F \colon M=M \otimes_R R \to
M \otimes_R  {}^1 \! R$ is injective for any $R$-module $M$. For
every ideal $I$ in $R$, and for every $q=p^e$, we denote by
$I^{[q]}$ the ideal generated by the $q^{\rm th}$ powers of all
elements of $I$.

If $M$ is an $R$-module, then we put $\F^e(M):={}^e \! R \otimes_R
M$. Hence in $\F^e(M)$ we have $u\otimes (ay)=ua^{p^e}\otimes y$ for
every $a\in R$. Note that the $e$-times iterated Frobenius map
$F_M^e \colon M \to \F^e(M)$ is an $R$-linear map. The image of $z
\in M$ via this map is denoted by $z^q:= F_M^e(z)$. If $N$ is a
submodule of $M$, then we denote by
 $N_M^{[q]}$ (or simply by $N^{[q]}$) the image of the canonical map
$\F^e(N) \to \F^e(M)$ (note that if $N=I$ is a submodule of $M=R$,
then this is consistent with our previous notation for $I^{[q]}$).

First, we recall the definitions of classical tight closure and
related notions. Our references for classical tight closure theory
and for F-rational rings are \cite{HH1} and\cite{FW}, respectively;
see also the book \cite{Hu2}.

\begin{defn}
Let $I$ be an ideal in a ring $R$ of characteristic $p>0$.
\renewcommand{\labelenumi}{(\roman{enumi})}
\begin{enumerate}
\item The \textit{Frobenius closure} $I^F$ of $I$ is defined as the ideal of $R$ consisting of
all elements $x \in R$ such that $x^q \in I^{[q]}$ for some $q=p^e$.
If $R$ is F-pure, then $J=J^F$ for all ideals $J \subseteq R$.
The \textit{tight closure} $I^*$ of $I$ is defined to be the ideal
of $R$ consisting of all elements $x \in R$ for which there exists
$c \in R^{\circ}$ such that $cx^q \in I^{[q]}$ for all large
$q=p^e$.

\item We say that $c \in R^{\circ}$ is a {\it test element} if for all ideals $J \subseteq R$ and all $x \in J^*$, we have $cx^q \in I^{[q]}$ for all $q=p^e \ge 1$.
Every excellent and reduced ring $R$ has a test element.

\item If $N \subseteq M$ are $R$-modules, then
the {\it tight closure} $N^{*}_M$ of $N$ in $M$ is defined to be the
submodule of $M$ consisting of all elements $z \in M$ for which
there exists $c \in R^{\circ}$ such that $cz^q \in N^{[q]}_M$ for
all large $q = p^e$. The \textit{test ideal} $\tau(R)$ of $R$ is
defined to be $\tau(R)=\displaystyle\bigcap_{M} \Ann_R(0^{*}_M)$,
where $M$ runs over all finitely generated $R$-modules. If $M=R/I$,
then $\Ann_R(0^{*}_M)=(I:I^*)$. That is, $\tau(R) J^* \subseteq J$
for all ideals $J \subseteq R$.
We say that $R$ is \textit{F-regular} if $\tau(R_P)=R_P$ for all prime ideals $P$ of $R$.

\item $R$ is called \textit{F-rational} if $J^*=J$ for every ideal $J \subseteq R$ generated by parameters.
If $R$ is an excellent equidimensional  local ring, then $R$ is
F-rational if and only if $I=I^*$ for \emph{some} ideal $I$
generated by a full system of parameters for $R$.
\end{enumerate}
\end{defn}

We now recall the definition of $\a^t$-tight closure and of the
generalized test ideal $\tau(\a^t)$. The reader is referred to
\cite{HY} for details.

\begin{defn}\label{ta}
Let $\a$ be a fixed ideal in a reduced ring $R$ of characteristic $p>0$ such
that $\a \cap R^{\circ} \ne \emptyset$, and let $I$ be an arbitrary
ideal in $R$.
\renewcommand{\labelenumi}{(\roman{enumi})}
\begin{enumerate}
\item
Let $N \subseteq M$ be $R$-modules. Given a rational number $t \ge 0$,
the {\it $\a^t$-tight closure} $N^{*\a^t}_M$
of $N$ in $M$ is defined to be the submodule of $M$ consisting of all
elements $z \in M$ for which there exists $c \in R^{\circ}$ such that
$cz^q\a^{\lceil tq \rceil} \subseteq N^{[q]}_M$ for all large $q = p^e$.

\item
The \textit{generalized test ideal} $\tau(\a^t)$ is defined to be
$\tau(\a^t)=\displaystyle\bigcap_{M} \Ann_R(0^{*\a^t}_M)$,
 where $M$ runs through all finitely generated $R$-modules.
 If $\a=R$, then the generalized test ideal $\tau(\a^t)$ is nothing but the test ideal $\tau(R)$.

\item Assume that $R$ is an F-regular ring and that $J$
is an ideal containing $\a$ in its radical. The \textit{F-jumping
exponent} of $\a$ with respect to $J$ is defined by
$$\xi^J(\a)=\sup\{c \in \R_{\ge 0} \mid \tau(\a^c) \not\subseteq J\}.$$
If $(R,\m)$ is local, then we call the smallest F-jumping exponent
$\xi^{\m}(\a)$ the \textit{F-pure threshold} of $\a$ and denote it
by $\mathrm{fpt}(\a)$.
\end{enumerate}
\end{defn}

In characteristic zero, one defines multiplier ideals and their
jumping exponents using resolution of singularities (see Ch.~9 in
\cite{La}). It is known that for a given ideal in characteristic
zero and for a given $t$, the reduction mod $p\gg 0$ of the
multiplier ideal $\J(\a^t)$ coincides with the generalized test
ideal $\tau(\a_p^t)$ of the reduction $\a_p$ of $\a$. Therefore the
F-jumping exponent $\xi^J(\a)$ is a characteristic $p$ analogue of
jumping exponent of multiplier ideals. We refer to \cite{BMS1},
\cite{HM}, \cite{HY},  \cite{MTW} and \cite{TW} for further
discussions.

\section{Basic properties of F-thresholds}

The F-thresholds are invariants of singularities of a given ideal
$\a$ in positive characteristic, obtained by comparing the powers of
$\a$ with the Frobenius powers of other ideals. They were introduced
and studied in \cite{MTW} in the case when we work in a regular
ring. In this section, we recall the definition of F-thresholds and
study their basic properties when the ring is not necessarily
regular.

Let $R$ be a Noetherian ring of dimension $d$ and of characteristic
$p>0$. Let $\a$ be a fixed proper ideal of $R$ such that $\a \cap
R^{\circ} \ne \emptyset$. To each ideal $J$ of $R$ such that
$\a\subseteq \sqrt{J}$, we associate an F-threshold as follows. For
every $q=p^e$, let
$$\nu_{\a}^J(q):=\max\{r\in\N\vert\a^r\not\subseteq J^{[q]}\}.$$
Since $\a\subseteq \sqrt{J}$, this is a nonnegative integer (if $\a
\subseteq J^{[q]}$, then we put $\nu_{\a}^J(q)=0$). We put
$$\pt_{+}^J(\a)=\limsup_{q\to\infty}\frac{\nu_{\a}^J(q)}{q}, \quad \pt_{-}^J(\a)=\liminf_{q \to \infty}\frac{\nu_{\a}^J(q)}{q}.$$
When $\pt_{+}^J(\a)=\pt_{-}^J(\a)$, we call this limit the
\emph{F-threshold} of the pair $(R,\a)$ (or simply of $\a$) with
respect to $J$, and we denote it by $\pt^J(\a)$.

\begin{rem}\label{rem1}
(1) (cf.\! \cite[Remark 1.2]{MTW}) One has
$$0 \le \pt_-^J(\a) \le \pt_+^J(\a) < \infty.$$
In fact, if $\a$ is generated by $l$ elements and if $\a^N\subseteq
J$, then
$$\a^{N(l(p^e-1)+1)}\subseteq(\a^{[p^e]})^{N}
=(\a^N)^{[p^e]}\subseteq J^{[p^e]}.$$ Therefore $\nu_{\a}^J(p^e)\leq
N(l(p^e-1)+1)-1$. Dividing by $p^e$ and taking the limit gives
$\pt_+^J(\a)\leq Nl$.

(2) Question~{1.4} in \cite{MTW} asked whether the F-threshold
$\pt^J(\a)$ is a rational number (when it exists). A positive answer
was given in \cite{BMS1} and \cite{BMS} for a regular F-finite ring,
essentially of finite type over a field, and for every regular
F-finite ring, if the ideal $\a$ is principal. For a proof in the
case of a principal ideal in a complete regular ring (that is not
necessarily F-finite), see \cite{KLZ}. However, this question
remains open in general.
\end{rem}

Recall that a ring extension $R\hookrightarrow S$ is \emph{cyclic
pure} if for every ideal $I$ in $R$, we have $IS\cap R=I$.

\begin{prop}[\textup{cf. \cite[Proposition 1.7]{MTW}}]\label{basic}
Let $\a$, $J$ be ideals as above.
\begin{enumerate}
\item If $I\supseteq J$, then $\pt_{\pm}^I(\a) \leq \pt_{\pm}^J(\a)$.
\item If $\b \subseteq \a$, then $\pt_{\pm}^J(\b) \leq\pt_{\pm}^J(\a)$.
Moreover, if $\a\subseteq \overline{\b}$, then $\pt_{\pm}^J(\b)
 = \pt_{\pm}^J(\a)$.
\item $\pt_{\pm}^J(\a^r) = \frac{1}{r} \pt_{\pm}^J(\a)$ for every integer $r \ge 1$.
\item  $\pt_{\pm}^{J^{[q]}}(\a) =q\pt_{\pm}^J(\a)$ for every $q=p^e$.
\item If $R \hookrightarrow S$ is a cyclic pure extension, then
$$\pt_{\pm}^J(\a)=\pt_{\pm}^{JS}(\a S).$$
\item Let $R \hookrightarrow S$ be an integral extension.
If the conductor ideal $\mathfrak{c}(S/R):=\Ann_R (S/R)$ contains
the ideal $\a$ in its radical, then
$$\pt_{\pm}^J(\a)=\pt_{\pm}^{JS}(\a S).$$
\item $\pt_+^J(\a) \le c$ $($resp. $\pt_-^J(\a) \ge c)$ if and only if for every power $q_0$ of $p$, we have
$\a^{\lceil cq \rceil+q/q_0} \subseteq J^{[q]}$ $($resp. $\a^{\lceil
cq \rceil-q/q_0} \not\subseteq J^{[q]})$ for all $q=p^e \gg q_0$.
\end{enumerate}
\end{prop}

\begin{proof}
For (1)--(4), see \cite{MTW} (the proofs therein do not use the fact
that $R$ is regular). If $R\hookrightarrow S$ is cyclic pure, then
$\nu^{JS}_{\a S}(q)=\nu^J_{\a}(q)$ for every $q$, and we get (5).

For (6), we fix a positive integer $m$ such that $\a^m \subseteq
\mathfrak{c}(S/R)$. By the definition of the conductor ideal
$\mathfrak{c}(S/R)$, if $(\a S)^n \subseteq (JS)^{[q]}$ for some $n
\in \N$ and some $q=p^e$, then $\a^{m+n} \subseteq J^{[q]}$.
This implies that
$$\nu_{\a S}^{JS}(q) \leq \nu_{\a}^{J}(q) \leq \nu_{\a S}^{JS}(q) +m.$$
These inequalities imply (6).

In order to prove (7), suppose first that $\pt^J_+(\a) \le c$. It
follows from the definition of $\pt^J_+(\a)$ that for every power
$q_0$ of $p$, we can find $q_1$ such that $\nu^J_\a(q)/q  <
c+\frac{1}{q_0}$ for all $q=p^e \ge q_1$. Thus, $\nu^J_{\a}(q) <
\lceil cq \rceil+\frac{q}{q_0}$, that is,
\begin{equation}\label{eq_prop_basic}
\a^{\lceil cq \rceil +q/q_0} \subseteq J^{[q]}
\end{equation}
 for all $q=p^e \ge
q_1$. Conversely, suppose that (\ref{eq_prop_basic}) holds for every
$q\geq q_1$. This implies $\nu^J_{\a}(q) \le \lceil cq \rceil
+\frac{q}{q_0}-1$. Dividing by $q$ and taking the limit gives
$\pt_+^J(\a)\leq c+\frac{1}{q_0}$. If this holds for every $q_0$, we
conclude that $\pt_+^J(\a)\leq c$. The assertion regarding
$\pt_-^J(\a)$ follows from a similar argument.
\end{proof}

We now give a variant of the definition of F-threshold. If $\a$ and
$J$ are ideals in $R$, such that $\a\cap R^{\circ}\neq\emptyset$ and
$\a\subseteq\sqrt{J}$, then we put
$$\widetilde{\nu}^J_{\a}(q):=\max \{r
\in \N \mid \a^r \not\subseteq (J^{[q]})^F\}.$$
It follows from the
 definition of
Frobenius closure that if $u\not\in (J^{[q]})^F$, then $u^p\not\in
(J^{[pq]})^F$. This means that
$$\frac{\widetilde{\nu}_{\a}^J(pq)}{pq} \ge \frac{\widetilde{\nu}_{\a}^J(q)}{q}$$
for all $q=p^e$. Thus,
$$\lim_{q \to \infty} \frac{\widetilde{\nu}_{\a}^J(q)}{q}=\sup_{q=p^e}\frac{\widetilde{\nu}_{\a}^J(q)}{q}.$$
We denote this limit by $\widetilde{\pt}^{J}(\a)$. Note that we have
$\widetilde{\pt}^J(\a)\leq\pt_{-}^J(\a)$.

The F-threshold $\pt^J(\a)$ exists in many cases.

\begin{lem}\label{lem1}
Let $\a$, $J$ be as above.
\begin{enumerate}
\item If $J^{[q]}=(J^{[q]})^F$ for all large $q=p^e$, then the F-threshold $\pt^J(\a)$ exists, that is, $\pt_{+}^J(\a)=\pt_{-}^J(\a)$. In particular, if $R$ is F-pure, then $\pt^J(\a)$ exists.
\item If the test ideal $\tau(R)$ contains $\a$ in its radical, then the F-threshold $\pt^J(\a)$ exists and $\pt^J(\a)=\pt^{J^*}(\a)$.
\item If $\a$ is principal, then $c^J(\a)$ exists.
\end{enumerate}
\end{lem}

\begin{proof}
(1) follows from the previous discussion since in that case we have
$\widetilde{\nu}^J_{\a}(q)=\nu^J_{\a}(q)$ for all $q\gg 0$.

In order to prove (2), we take an integer $m \ge 1$ such that $\a^m
\subseteq \tau(R)$. Then, by the definition of $\tau(R)$, one has
$\a^{2m} ((J^*)^{[q]})^F  \subseteq \a^m (J^*)^{[q]} \subseteq
J^{[q]}$ for all $q=p^e$. This means that
$$\widetilde{\nu}_{\a}^{J^*}(q) \le \nu_{\a}^{J^*}(q) \le \nu_{\a}^J(q) \le \widetilde{\nu}_{\a}^{J^*}(q)+2m.$$
Since $\widetilde{\pt}^{J^*}(\a)$ always exists, $\pt^J(\a)$ and
$\pt^{J^*}(\a)$ also exist and these three limits are all equal.

For (3), note that if $\a$ is principal and $\a^r\subseteq J^{[q]}$,
then $a^{pr}\subseteq J^{[pq]}$. Therefore we have
$$\frac{\nu^J_{\a}(pq)+1}{pq}\leq\frac{\nu^J_{\a}(q)+1}{q}$$
for every $q=p^e$. This implies that
$$\lim_{q\to\infty}\frac{\nu^J_{\a}(q)}{q}=\lim_{q\to\infty}\frac{\nu^J_{\a}(q)+1}{q}
=\inf_{q=p^e}\frac{\nu^J_{\a}(q)}{q}.$$

\end{proof}

As shown in \textup{\cite[Proposition 2.7]{MTW}}, the F-threshold
$\pt^J(\a)$ coincides with the F-jumping exponent $\xi^J(\a)$ when
the ring is $F$-finite and regular. The statement in \emph{loc.
cit.} requires the ring to be local, however the proof easily
generalizes to the non-local case (see \cite{BMS}). More precisely,
we have the following

\begin{prop}\label{regular}
Let $R$ be an F-finite regular ring of characteristic $p>0$. If $\a$
is a nonzero ideal contained in the radical of $J$, then
$\tau(\a^{\pt^J(\a)}) \subseteq J$. Going the other way, if $\alpha
\in \R_{+}$, then $\a$ is contained in the radical of
$\tau(\a^{\alpha})$ and $\pt^{\tau(\a^{\alpha})}(\a) \leq \alpha$.
In particular, the F-threshold $\pt^J(\a)$ coincides with the
F-jumping exponent $\xi^J(\a)$.
\end{prop}

\begin{rem}\label{remark_equality}
The F-threshold $\pt^J(\a)$ sometimes coincide with the F-jumping
exponent $\xi^J(\a)$ even when $R$ is singular. For example, let
$R=k[[X,Y,Z,W]]/(XY-ZW)$, and let  $\m$ be the maximal ideal of $R$.
Then the F-threshold $\pt^{\m}(\m)$ of $\m$ with respect to $\m$ and
the F-pure threshold (that is, the smallest F-jumping exponent)
$\mathrm{fpt}(\m)$ of $\m$ are both equal to two.

However, $\pt^J(\a)$ does not agree with $\xi^J(\a)$ in general. For
example, let $R=k[[X,Y,Z]]/(XY-Z^2)$ be a rational double point of
type $A_1$ over a field $k$ of characteristic $p >2$ and let $\m$ be
the maximal ideal of $R$. Then $\mathrm{fpt}(\m)=1$ (see
\cite[Example 2.5]{TW}), whereas $\pt^{\m}(\m)=3/2$.
\end{rem}

\begin{rem}\label{localization}
Suppose that $\m$ is a maximal ideal in any Noetherian ring $R$, and
that $J$ is an $\m$-primary ideal. For every $q=p^e$ we have
$J^{[q]}R_{\m}\cap R=J^{[q]}$, hence for every ideal $\a\subseteq\m$
we have $\nu^J_{\a}(q)=\nu^{JR_{\m}}_{\a R_{\m}}(q)$. In particular,
$\pt^J_{\pm}(\a)=\pt^{JR_{\m}}_{\pm}(\a R_{\m})$.
\end{rem}

\begin{eg}\label{maximal-threshold}
\begin{enumerate}
\item[(i)] Let $R$ be a Noetherian local ring of characteristic $p>0$,
and let $J=(x_1,\ldots,x_d)$, where $x_1,\ldots,x_d$ form a full
system of parameters in $R$. It follows from the Monomial Conjecture
(which is a theorem in this setting, see \cite[Prop. 3]{Ho}) that
$(x_1\cdots x_d)^{q-1}\not\in J^{[q]}$ for every $q$. Hence
$\nu^J_J(q)\geq d(q-1)$ for every $q$, and therefore
$\pt_{-}^J(J)\geq d$. On the other hand, $\pt_+^J(J)\leq d$ by
Remark~\ref{rem1} (1), and we conclude that $\pt^J(J)=d$.

\item[(ii)] Let  $R=k[x_1, \dots, x_d]$ be a $d$-dimensional polynomial ring
over a field $k$ of characteristic $p>0$, and let $\a, J \subseteq
R$ be zero-dimensional ideals generated by monomials. In order to
compute $c^J(\a)$ we may assume that $k$ is perfect, hence we may
use Proposition~\ref{regular}.

Let $P(\a)\subseteq \R_{\ge 0}^d$ denote the Newton polyhedron of
$\a$, that is $P(\a)$ is the convex hull of those
$u=(u_1,\ldots,u_n)\in\N^n$ such that $x^u=x_1^{u_1}\cdots
x_n^{u_n}\in\a$. It follows from \cite[Thm. 6.10]{HY} that
$$\tau(\a^{c})=(x^u\mid u+e\in {\rm Int}(c\cdot P_{\a})),$$
where $e=(1,1,\ldots,1)$. We deduce that if $\lambda(u)$ is defined
by the condition $u+e\in \partial(\lambda(u)\cdot P(\a))$, then
$$\pt^J(\a)=\max\{\lambda(u)\mid u\in\N^n, x^u\not\in J\}$$
(note that since $J$ is zero-dimensional, this maximum is over a
finite set). In particular, we see that if
$J=(x_1^{a_1},\ldots,x_n^{a_n})$, then $\pt^J(\a)$ is characterized
by $a=(a_1,\ldots,a_n)\in\partial(c^J(\a)\cdot P(\a))$.

\item[(iii)] Let $(R,\m)$ be a $d$-dimensional regular local ring of
characteristic $p>0$, and let $J \subset R$ be an $\m$-primary
ideal. We claim that
\begin{equation}\label{eq_example}
\pt^J(\m)=\max\{r \in \Z_{\ge 0} \mid \m^r \not\subseteq J\}+d.
\end{equation}
 In particular, $\pt^J(\m)$ is an integer $\geq d$.

Indeed, if $u\not\in J$, then $(J\colon u)\subseteq\m$, hence
$J^{[q]}\colon u^q=(J\colon u)^{[q]}\subseteq\m^{[q]}$, and
therefore $u^q\m^{d(q-1)}\not\subseteq J^{[q]}$. If $u\in\m^r$, it
follows that $\nu^J_{\m}(q)\geq rq+d(q-1)$. Dividing by $q$ and
passing to the limit gives $c^J(\m)\geq r+d$, hence we have "$\geq$"
in (\ref{eq_example}). For the reverse inequality, note that if
$\m^{r+1}\subseteq J$, then
$$\m^{(r+d)q}\subseteq(\m^{r+1})^{[q]}\subseteq J^{[q]}$$
for every $q=p^e$. Hence $\nu^J_{\m}(q)\leq (r+d)q-1$ for all $q$,
and we get $c^J(\m)\leq r+d$.
\end{enumerate}
\end{eg}

\section{Connections with tight closure and integral closure}

\begin{thm}\label{tight}
Let $(R,\m)$ be an excellent analytically irreducible Noetherian
local domain of positive characteristic $p$. Set $d=\dim(R)$, and
let $J=(x_1, \dots, x_d)$ be an ideal generated by a full system of
parameters in $R$, and let $I\supseteq J$ be another ideal. Then $I$
is not contained in the tight closure $J^*$ of $J$ if and only if
there exists $q_0=p^{e_0}$ such that $x^{q_0-1} \in I^{[q_0]}$,
where $x=x_1x_2 \cdots x_d$.
\end{thm}
\begin{proof}
After passing to completion, we may assume that $R$ is a complete
local domain. Suppose first that $x^{q_0-1} \in I^{[q_0]}$, and by
way of contradiction suppose also that $I \subseteq J^*$. Let $c \in
R^{\circ}$ be a test element. Then for all $q=p^e$, one has
$cx^{q(q_0-1)} \in cI^{[qq_0]} \subset J^{[qq_0]}$, so that $c \in
J^{[qq_0]}:x^{q(q_0-1)} \subseteq (J^{[q]})^*$, by colon-capturing \cite[Theorem 7.15a]{HH1}.
Therefore $c^2$ lies in $\bigcap_{q=p^e}J^{[q]}=(0)$, a
contradiction.

Conversely, suppose that $I \nsubseteq J^*$, and choose an element
$f \in I \smallsetminus J^*$. We choose a coefficient field $k$, and
let $B=k[[x_1, \dots, x_d,f]]$ be the complete subring of $R$
generated by $x_1, \dots, x_d, f$. Note that $B$ is a hypersurface
singularity, hence Gorenstein. Furthermore, by persistence of tight
closure \cite[Lemma 4.11a]{HH1}, $f \notin ((x_1, \dots, x_d)B)^*$. If we prove that there
exists $q_0=p^{e_0}$ such that $x^{q_0-1} \in ((x_1, \dots, x_d,
f)B)^{[q_0]}$, then clearly $x^{q_0-1}$ is also in $I^{[q_0]}$.
Hence we can reduce to the case in which $R$ is Gorenstein. Since $I
\not\subseteq J^*$, it follows from a result of Aberbach \cite{Ab}
that $J^{[q]}\colon I^{[q]} \subseteq \m^{n(q)}$, where $n(q)$ is a
positive integer with $\lim_{q\to\infty}n(q)=\infty$. In particular,
we can find $q_0=p^{e_0}$ such that $J^{[q_0]}:I^{[q_0]} \subseteq
J$. Therefore $x^{q_0-1} \in J^{[q_0]}\colon J \subseteq
J^{[q_0]}\colon (J^{[q_0]}\colon I^{[q_0]})=I^{[q_0]}$, where the
last equality follows from the fact that $R$ is Gorenstein.
\end{proof}

\begin{cor}\label{cor_tight}
Let $(R,\m)$ be a $d$-dimensional excellent analytically irreducible
Noetherian local domain of characteristic $p>0$, and let $J=(x_1,
\dots, x_d)$ be an ideal generated by a full system of parameters in
$R$. Given an ideal $I \supseteq J$, we have $I \subseteq J^*$ if
and only if $\pt^I_+(J)=d$ (and in this case $c^I(J)$ exists). In
particular, $R$ is F-rational if and only if $\pt^I_+(J) < d$ for
every ideal $I \supsetneq J$.
\end{cor}
\begin{proof}
Note first that by Remark~\ref{rem1} (1), for every $I\supseteq J$
we have $\pt_+^J(I)\leq d$. Suppose now that $I\subseteq J^*$. It
follows from Theorem~\ref{tight} that $J^{d(q-1)}\not\subseteq
I^{[q]}$ for every $q=p^e$. This gives $\nu^I_J(q)\geq d(q-1)$ for
all $q$, and therefore $\pt^I_{-}(J)\geq d$. We conclude that in
this case $\pt^I_+(J)=\pt^I_{-}(J)=d$.

Conversely, suppose that $I \not\subseteq J^*$. By Theorem
\ref{tight}, we can find $q_0=p^{e_0}$ such that
$$\b:=(x_1^{q_0},\ldots,x_d^{q_0},(x_1\cdots x_d)^{q_0-1})\subseteq I^{[q_0]}.$$
If $(x_1,\ldots,x_d)^r\not\subseteq\b^{[q]}$, then
$$r\leq (qq_0-1)(d-1)+q(q_0-1)-1=qq_0d-q-d.$$
Therefore $\nu^{\b}_J(q)\leq qq_0d-q-d$ for every $q$, which implies
$\pt^{\b}(J)\leq q_0d-1$. Since $q_0$ is a fixed power of $p$, we
deduce
$$\pt_+^I (J) =\frac{1}{q_0}\pt_+^{I^{[q_0]}}(J) \le \frac{1}{q_0} \pt^{\b}(J)\leq d-\frac{1}{q_0}<d.$$
\end{proof}

\begin{thm}\label{integral}
Let $(R, \m)$ be a $d$-dimensional formally equidimensional
Noetherian local ring of characteristic $p>0$. If $I$ and $J$ are
ideals in $R$, with $J$ generated by a full system of parameters,
then
\begin{enumerate}
\item $\pt_+^J(I)\leq d$ if and only if $I\subseteq \overline{J}$.
\item If, in addition, $J\subseteq I$, then $I\subseteq\overline{J}$ if and only
if $\pt^J_+(I)=d$. Moreover, if these equivalent conditions hold,
then $\pt^J(I)=d$.
\end{enumerate}
\end{thm}

\begin{proof}
Note that if $J\subseteq I$, then
$\pt_{-}^J(I)\geq\pt_{-}^J(J)=\pt^J(J)=d$, by Example
\ref{maximal-threshold} (i). Hence both assertions in (2) follow
from the assertion in (1).

One implication in (1) is easy: if $I \subseteq \overline{J}$, then
by Proposition \ref{basic} (2) we have
$\pt_+^J(I)\leq\pt_+^J(\overline{J})= \pt^J(J)=d$. Conversely,
suppose that $\pt_+^J(I)\leq d$. In order to show that
$I\subseteq\overline{J}$, we may assume that $R$ is complete and
reduced. Indeed, first note that the inverse image of
 $\overline{J\widehat{R}_{\rm red}}$ in $R$ is contained in
 $\overline{J}$, hence it is enough to show that $I\widehat{R}_{\rm
 red}\subseteq\overline{J \widehat{R}_{\rm red}}$. Since
$J\widehat{R}_{\rm red}$ is again generated by a full system of
parameters, and since we trivially have
$$\pt^{J\widehat{R}_{\rm red}}(I\widehat{R}_{\rm
red})\leq \pt^J(I)\leq d,$$  we may replace $R$ by $\widehat{R}_{\rm
red}$.

 Since $R$ is complete and reduced, we can find a test element $c$
for $R$. By Proposition~\ref{basic} (7), the assumption
$\pt_+^J(I)\leq d$ implies that for all $q_0=p^{e_0}$ and for all
large $q=p^e$, we have
 $$I^{q(d+(1/q_0))} \subseteq J^{[q]}.$$
Hence $I^q J^{q(d-1+(1/q_0))} \subseteq J^{[q]}$, and thus
$$I^q \subseteq J^{[q]}\colon J^{q(d-1+(1/q_0))} \subseteq (J^{q-d+1-(q/q_0)})^*,$$
where the last containment follows from the colon-capturing property
of tight closure \cite[Theorem 7.15a]{HH1}.
We get $cI^q \subseteq cR \cap J^{q-d+1-(q/q_0)}
\subseteq cJ^{q-d+1-(q/q_0)-l}$ for some fixed integer $l$ that is
independent of $q$, by the Artin-Rees lemma. Since $c$ is a non-zero
divisor in $R$, it follows that
\begin{equation}\label{eq_integral}
I^q \subseteq J^{q-d+1-(q/q_0)-l}.
\end{equation}
 If $\nu$ is a discrete valuation
with center in $\m$, we may apply $\nu$ to (\ref{eq_integral}) to
deduce $q\nu(I) \ge \left(q-d+1-\frac{q}{q_0}-l\right) \nu(J)$.
Dividing by $q$ and letting $q$ go to infinity gives $\nu(I) \ge
\left(1-\frac{1}{q_0}\right)\nu(J)$. We now let $q_0$ go to infinity
to obtain $\nu(I) \ge \nu(J)$. Since this holds for every $\nu$, we
have $I\subseteq\overline{J}$.
\end{proof}

\begin{eg}  Let $(R,\m)$ be a regular local ring of characteristic $p>0$
with $\dim(R)=d$, and $J$ be an ideal of $R$ generated by a full
system of parameters. We define $a$ to be the  maximal integer $n$
such that $\m^n \not\subseteq J$. Then $\m^s\subseteq\overline{J}$
if and only if $s\ge \frac{a}{d}+1$ since
$\pt^J(\m^s)=\frac{a+d}{s}$ by Example~\ref{maximal-threshold} (iii)
and Proposition 2.2 (3).
\end{eg}

\begin{ques}
Does this statement hold in a more general setting ? Can we replace
\lq\lq regular" by \lq\lq Cohen-Macaulay" ?
\end{ques}

\section{F-thresholds of modules}

In the section we give a generalization of the notion of
F-thresholds, in which we replace the auxiliary ideal in the
definition by a submodule of a given module. We have seen in
Proposition~\ref{regular} that in a regular F-finite ring, the
F-thresholds of an ideal $\a$ coincide with the F-jumping exponents
of $\a$. This might fail in non-regular rings, and in fact, it is
often the case that $\mathrm{fpt}(\a)<\pt^J(\a)$ for every ideal
$J$. However, as Corollary~\ref{cor-F-thr} below shows, we can
remedy this situation if we consider the following more general
notion of F-thresholds.

Suppose now that $\a$ is a fixed ideal in a Noetherian ring $R$ of
characteristic $p>0$. Let $M$ be an $R$-module, and $N\subseteq M$ a
submodule such that $\a^n N=0$ for some $n>0$. We define
\begin{enumerate}
\item For $q=p^e$, let  $\nu_{M,\a}^{N}(q)=
 \max\{ r\in\N \mid  \a^r N_M^{[q]}\neq 0\}$
(we put $\nu_{M,\a}^N(q)=0$ if $\a N_M^{[q]}=0$).
\item $\pt_{M,+}^N(\a)=\limsup_{q\to\infty}\frac{\nu_{M,\a}^N(q)}{q}$ and
$\pt_{M,-}^N(\a)=\liminf_{q \to \infty}\frac{\nu_{M,\a}^N(q)}{q}.$
When $\pt_{M,+}^N(\a)=\pt_{M,-}^N(\a)$, we call this limit the
\emph{F-threshold} of  $\a$  with respect to $(N,M)$, and we denote
it by $\pt^N_M(\a)$.
\end{enumerate}

\begin{rem} If $J$ is an ideal of $R$ with $\a\subseteq \sqrt{J}$, then
it is clear that $\nu_{\a,A/J}^{A/J}(q) = \nu_{\a}^J(q)$, hence
$\pt_{A/J,\pm}^{A/J}(\a)=\pt_{\pm}^{J}(\a)$.  Thus the notion of
F-threshold with respect to modules extends our previous definition
of F-thresholds with respect to ideals.
\end{rem}

\begin{lem}\label{elem-mod}
 Let $R$, $\a$, $M$ and $N$ be as in the above definition.
\begin{enumerate}
\item If $\b\subseteq \a$ is an ideal, then $\pt_{M,\pm}^N(\b)\le \pt_{M,\pm}^N(\a)$.
\item If $N'\subseteq N$, then $\pt_{M,\pm}^{N'}(\a)\le \pt_{M,\pm}^N(\a)$.
\item If $\phi\colon M\to M'$ is a homomorphism of $R$-modules, and if
$N'=\phi(N)$, then $\pt_{M',\pm}^{N'}(\a)\le \pt_{M,\pm}^N(\a)$. If
$R$ is regular and $\phi$ is injective, then $\pt_{M',\pm}^{N'}(\a)=
\pt_{M,\pm}^N(\a)$.
\item If $R$ is F-pure, then $\dfrac{\nu_{M,\a}^N(q)}{q}\le
\dfrac{\nu_{M,\a}^N(qq')}{qq'}$ for every $q,q'$. Hence in this case
the limit  $\pt_{M}^N(\a)$ exists and it is equal to
$\sup_q\dfrac{\nu_{M,\a}^N(q)}{q}$.
\end{enumerate}
\end{lem}

\begin{proof} The assertions in (1) and (2) follow from definition.
For (3), note that $\phi$ induces a surjection $N^{[q]}\to
{N'}^{[q]}$, which gives the first statement. Moreover, if $R$ is
regular and $\phi$ is injective, then the flatness of the Frobenius
morphism implies $N^{[q]}\simeq {N'}^{[q]}$, and we have equality.

Suppose now that $R$ is F-pure, hence $M\otimes_R {}^e \! R$ is a
submodule of $M \otimes_R {}^{ee'} \! R$. If $q=p^e$ and
$q'=p^{e'}$, and if $\a^r N^{[q]}\ne 0$, then $\a^{q'r}N^{[qq']}
\supseteq (\a^{r})^{[q']}N^{[qq']}\ne 0$. Therefore
$\nu^N_{M,\a}(qq')\geq q'\cdot\nu^N_{M,\a}(q)$.
\end{proof}

Our next proposition gives an analogue of Proposition~\ref{regular}
in the non-regular case.

\begin{prop}\label{F-thr-mod2}
Let $\a$ be a proper nonzero ideal in a local normal $\Q$-Gorenstein
ring $(R,\m)$. Suppose that $R$ is F-finite and F-pure, and that the
test ideal $\tau(R)$ is $\m$-primary. We denote by $E$ the injective
hull of $R/\m$.
\begin{enumerate}
\item If $N$ is a submodule of $E$ such that $\a\subseteq\sqrt{{\rm Ann}_R(N)}$,
and if $\alpha=\pt^N_E(\a)$, then $N \subseteq
(0)_E^{*\a^{\alpha}}$.
\item If $\alpha$ is a non-negative real number, and if we put $N=
(0)_E^{*\a^{\alpha}}$, then $\pt^N_E(\a)\le \alpha$.
\item There is an order-reversing bijection between
the F-thresholds of $\a$ with respect to the submodules of $E$ and
the ideals of the form $\tau(\a^{\alpha})$.
\end{enumerate}
\end{prop}

\begin{proof}
For (1), note that since $R$ is F-pure, we have
$\nu^N_E(q)\leq\alpha q$ for every $q=p^e$. This implies
$$\a^{\lceil \alpha q\rceil+1}N_E^{[q]}=0,$$
hence for every nonzero $d\in\a$ we have $d\a^{\lceil\alpha
q\rceil}N_E^{[q]}=0$ for all $q$. By definition, $N \subseteq
(0)_E^{*\a^{\alpha}}$.

Suppose now that $\alpha\geq 0$, and that $N=(0)_E^{*\a^{\alpha}}$.
By hypothesis, we can find $m$ such that $\a^m\subseteq\tau(R)$. It
follows from \cite[Cor. 2.4]{HT} that every element in $\tau(R)$ is
an $\a^{\alpha}$-test element. Therefore $\a^{m+\lceil \alpha
q\rceil} N_E^{[q]}=0$, hence $\nu^N_{E,\a}(q)<m+\alpha q$ for all
$q\gg 0$. Dividing by $q$ and taking the limit as $q$ goes to
infinity, gives $\pt^N_E(\a)\leq\alpha$.

We assume that $R$ is F-finite, normal and $\Q$-Gorenstein, hence
for every non-negative $t$ we have $\tau(\a^t)={\rm
Ann}_R(0_E^{*\a^t})$. Note also that by \cite[Prop. 3.2]{HT}, taking
the generalized test ideal commutes with completion. This shows that
the set of ideals of the form $\tau(\a^{\alpha})$ is in bijection
with the set of submodules of $E$ of the form
$(0)_E^{*\a^{\alpha}}$. Hence in order to prove (3) it is enough to
show that the map
$$\{(0)_E^{*\a^{\alpha}}\mid\alpha\geq 0\}\to
\{\pt^N_E(\a)\mid N\subseteq E, \a\subseteq \sqrt{{\rm
Ann}_R(N)}\}$$ that takes $N$ to $\pt^N_E(\a)$ is bijective, the
inverse map taking $\alpha$ to $(0)_E^{*\a^{\alpha}}$.

Suppose first that $N=(0)_E^{*\a^{\alpha}}$, and let
$\beta=\pt^N_E(\a)$. It follows from (2) that $\beta\leq\alpha$,
hence $(0)_E^{*\a^{\beta}}\subseteq N$. On the other hand, (1) gives
$N\subseteq (0)_E^{*\a^{\beta}}$, hence we have equality.

Let us now start with $\alpha=\pt^N_E(\a)$, and let $N'=
(0)_E^{*\a^{\alpha}}$. We deduce from (1) that $N\subseteq N'$,
hence $\pt^{N'}_E(\a)\geq \alpha$. Since (2) implies
$\pt^{N'}_E(\a)\leq\alpha$, we get $\alpha=\pt^{N'}_E(\a)$, which
completes the proof of (3).
\end{proof}

\begin{cor}\label{cor-F-thr}
Let $\a$ be a proper nonzero ideal in a local normal $\Q$-Gorenstein
ring $(R,\m)$. If $R$ is F-finite and F-regular, then for every
ideal $J$ in $R$ we have
$$\xi^J(\a)=\pt_E^N(\a),$$
where $E$ is the injective hull of $R/\m$ and $N={\rm Ann}_E(J)$. In
particular, the F-pure threshold $\mathrm{fpt}(\a)$ is equal to
$\pt^Z_E(\a)$, where $Z=(0\,\colon_E\,\m)$ is the socle of $E$.
\end{cor}

\begin{proof}
Let $\beta:=\pt^N_E(\a)$. Given $\alpha\geq 0$, Matlis duality
implies that $\tau(\a^{\alpha})\subseteq J$ if and only if
$N\subseteq (0)_E^{*\a^{\alpha}}$. If this holds, then part (2) in
the proposition gives
$$\alpha\geq
\pt_E^{(0)_E^{*\a^{\alpha}}}(\a)\geq \pt_E^N(\a)=\beta.$$
Conversely, if $\alpha\geq\beta$, then
$$(0)_E^{*\a^{\alpha}}\supseteq (0)_E^{*\a^{\beta}}\supseteq
N,$$ by part (1) in the proposition. This shows that
$\pt^N_E(\a)=\xi^J(\a)$, and the last assertion in the corollary
follows by taking $J=\m$.
\end{proof}

\begin{rem}
Let $\a$ be an ideal in the local ring $(R,\m)$. We have seen that
$\pt^I(\a)\ge \pt^{\m}(\a)$ for every proper ideal $I$. Note also
that applying Prop~\ref{elem-mod} (3) to the embedding $R/m\simeq
Z\hookrightarrow E=E_R(R/\m)$, we get
$\pt^{\m}(\a)=\pt^{R/\m}_{R/\m}(\a)\ge
\pt^{Z}_E(\a)=\mathrm{fpt}(\a)$. Thus we always have
$\mathrm{fpt}(\a)\le \pt^I(\a)$, and equality is possible only if
$\mathrm{fpt}(\a)= \pt^{\m}(\a)$. While this equality holds in some
non-regular examples (see Remark~\ref{remark_equality}), this seems
to happen rather rarely.
\end{rem}

\section{Connections between F-thresholds and multiplicity}

Given an $\m$-primary ideal $\a$ in a regular local ring $(R,\m)$,
essentially of finite type over a field of characteristic zero,  de
Fernex, Ein and the second author proved in \cite{dFEM} an
inequality involving the log canonical threshold $\mathrm{lct}(\a)$
and the multiplicity $e(\a)$. Later, the third and fourth authors
gave in \cite{TW} a characteristic $p$ analogue of this result,
replacing the log canonical threshold $\mathrm{lct}(\a)$ by the
F-pure threshold $\mathrm{fpt}(\a)$. We propose the following
conjecture, generalizing this inequality.

\begin{conj}\label{mult conj}
Let $(R,\m)$ be a $d$-dimensional Noetherian local ring of
characteristic $p>0$. If $J \subseteq \m$ is an ideal generated by a
full system of parameters, and if $\a \subseteq\m$ is an
$\m$-primary ideal, then
$$e(\a) \ge \left(\frac{d}{\pt_{-}^J(\a)}\right)^d e(J).$$
\end{conj}

\begin{rem}\label{rmkmult}
(1) When $R$ is regular and $J=\m$, the above conjecture is
precisely the above-mentioned inequality, see \cite[Proposition
4.5]{TW}.

(2) When $R$ is a $d$-dimensional regular local ring, essentially of
finite type over a field of characteristic zero, we can consider an
analogous problem:  let $\a, J$ be $\m$-primary ideals in $R$ such
that $J$ is generated by a full system of parameters. Does the
following inequality hold
$$e(\a) \ge \left(\frac{d}{\lambda^J(\a)}\right)^d e(J),$$
where $\lambda^J(\a):=\max\{c>0\mid\J(\a^{c})\not\subseteq J\}$.
This would generalize the inequality in \cite{dFEM}, which is the
special case $J=\m$. However, this version is also open in general.

(3) The condition in Conjecture~\ref{mult conj} that $J$ is
generated by a system of parameters is crucial, as otherwise there
are plenty of counterexamples. Suppose, for example, that $(R,\m)$
is a regular local ring of dimension $d \ge 2$ and of characteristic
$p>0$.  Let $\a=\m^k$ and $J=\m^{\ell}$ with $k\geq 1$, $\ell\geq 2$
integers. It follows from Example~\ref{maximal-threshold} (3) that
$\pt^J(\a)=(d+\ell-1)/k$. Moreover, we have $e(\a)=k^d$ and
$e(J)=\ell^d$, thus $$e(\a)=k^d < (dk\ell/(d+\ell-1))^d=
\left(\frac{d}{\pt^J(\a)}\right)^d e(J).$$

\end{rem}

\begin{eg}
Let $R=k\llbracket X,Y,Z\rrbracket/(X^2+Y^3+Z^5)$ be a rational
double point of type $E_8$, with $k$ a field of characteristic
$p>0$. Let $\a=(x,z)$ and $J=(y,z)$. Then $e(\a)=3$ and $e(J)=2$. It
is easy to check that $\pt^J(\a)=5/3$ and $\pt^{\a}(J)=5/2$. Thus,
\begin{align*}
&e(\a)=3 > \frac{72}{25}=\left(\frac{2}{\pt^J(\a)}\right)^2 e(J),\\
&e(J)=2>\frac{48}{25}=\left(\frac{2}{\pt^{\a}(J)}\right)^2 e(\a).
\end{align*}
See Corollary~\ref{cor_homogeneous} below for a general statement in
the homogeneous case.
\end{eg}

We now show that Conjecture \ref{mult conj} implies an effective
estimate of the multiplicity of complete intersection F-rational
rings.

\begin{prop}
Let $(R,\m)$ be a $d$-dimensional F-rational local ring of
characteristic $p>0$ with infinite residue field $($resp. a rational
singularity over a field of characteristic zero$)$ which is a
complete intersection. If Conjecture $\ref{mult conj}$ $($resp.
Remark $\ref{rmkmult}$ $(1))$ holds true for the regular case, then
$e(R) \le 2^{d-1}$.
\end{prop}
\begin{proof}
Let $J \subseteq \m$ be a minimal reduction of $\m$. Note that $J$
is generated by a full system of parameters for $R$. The Brian\c
con-Skoda theorem for F-rational rings (or for rational
singularities), see \cite{HV} and \cite{AH}, gives $\m^d \subseteq
J$. Taking the quotient of $R$ by $J$, we reduce the assertion in
the proposition to the following claim:
\begin{cl}
Let $(A,\m)$ be a complete intersection Artinian local ring of
characteristic $p>0$ (resp. essentially of finite type over a field
of characteristic zero). If $s$ is the largest integer $s$ such that
$\m^s \neq 0$, then $e(A) \le 2^s$.
\end{cl}

We now show that the regular case of Conjecture~\ref{mult conj}
implies the claim in positive characteristic (the argument in
characteristic zero is entirely analogous).
 Write $A=S/I$, where $(S,\n)$ is an $n$-dimensional
regular local ring and $I \subseteq S$ is an ideal generated by a
full system of parameters $f_1, \dots, f_n$ for $S$. For every $i$,
we denote by $\alpha_i$ the order of $f_i$. We may assume that
$\alpha_i \ge 2$ for all $i$.

Let $\n=(y_1,\ldots,y_n)$, and let us  write $f_i=\sum_ja_{ij}y_j$.
A standard argument relating the Koszul complexes on the $f_i$ and,
respectively, the $y_i$, shows that ${\rm det}(a_{ij})$ generates
the socle of $A$. In particular, if
$$s:=\max\{r \in \N \mid \n^r \not\subseteq I\},$$
then  $s\ge \sum_{i=1}^n (a_i-1) \ge n$. On the other hand, it
follows from Example \ref{maximal-threshold} (iii) that
$\pt^I(\m)=s+n$ (the corresponding formula in characteristic zero is
an immediate consequence of the description of the multiplier ideals
of the ideal of a point). Applying Conjecture \ref{mult conj} to
$S$, we get
$$1=e(\n)\ge \left(\frac{n}{\pt^I(\m)} \right)^n e(I)=\left(\frac{n}{s+n} \right)^n e(I).$$
Note that $(n/(s+n))^n \ge (s/(s+s))^s=(1/2)^s$, because $s \ge n$.
Thus, we have $e(A)=e(I) \le 2^s$.
\end{proof}

\begin{prop}
If $(R,\m)$ is a one-dimensional analytically irreducible local
domain of characteristic $p>0$, and if $\a, J$ are $\m$-primary
ideals in $R$, then $$\pt^J(\a)=\frac{e(J)}{e(\a)}.$$ In particular,
Conjecture $\ref{mult conj}$ holds in $R$.
\end{prop}
\begin{proof}
By Proposition \ref{basic} (5), we may assume that $R$ is a complete
local domain. Since $R$ is one-dimensional,  the integral closure
$\overline{R}$ is a DVR.  Therefore we have
 $$\pt^{J \overline{R}}(\a
\overline{R})=\ord_{\overline{R}}(J\overline{R})/\ord_{\overline{R}}(\a
\overline{R}).$$
On the other hand, $e(J
\overline{R})=\ord_{\overline{R}}(J\overline{R})$ and $e(\a
\overline{R})=\ord_{\overline{R}}(\a \overline{R})$. Thus, by
Proposition \ref{basic} (6),
$$\pt^J(\a)=\pt^{J \overline{R}}(\a \overline{R})=\frac{e(J \overline{R})}{e(\a \overline{R})}=\frac{e(J)}{e(\a)}.$$
\end{proof}

\begin{thm}\label{diagonal}
If $(R,\m)$ is a regular local ring of characteristic $p>0$ and
$J=(x_1^{a_1}, \dots, x_d^{a_d})$, with $x_1, \dots, x_d$ a full
regular system of parameters for $R$, and with $a_1, \dots, a_d$
positive integers, then the inequality given by
Conjecture~$\ref{mult conj}$ holds.
\end{thm}

\begin{proof}
The proof follows the idea in \cite{dFEM} and \cite{TW}, reducing
the assertion to the case when $\a$ is a monomial ideal, and then
using the explicit description of the invariants involved. We have
by definition $e(\a)=\lim_{n \to \infty}\frac{d! \cdot
\ell_R(R/\a^n)}{n^d}$, hence it is enough to show that for every
$\m$-primary ideal $\a$ of $R$,
\begin{equation}\label{eq_diagonal}
\ell_R(R/\a) \ge \frac{1}{d!} \left(\frac{d}{\pt^J(\a)} \right)^d
e(J).
\end{equation}
 After passing to completion and using Proposition~\ref{basic}
(5) and Remark~\ref{localization}, we see that it is enough to prove
the inequality (\ref{eq_diagonal}) in the case when $R=k[x_1, \dots,
x_d]$, $\m=(x_1,\ldots,x_d)$, $\a$ is $\m$-primary, and
$J=(x_1^{a_1}, \dots, x_d^{a_d})$.

Note that $e(J)=a_1 \cdots a_d$. We fix a monomial order $\lambda$
on the monomials in the polynomial ring, and use it to take a
Gr\"{o}bner deformation of $\a$, see \cite[Ch. 15]{eisenbud}. This
is a flat family $\{\a_s\}_{s \in k}$ such that $R/\a_s \cong R/\a$
for all $s \ne 0$, and such that $\a_0=\mathrm{in}_{\lambda}(\a)$,
the initial ideal of $\a$.

If $I$ is an ideal generated by monomials, we denote by $P(I)$ the
Newton polyhedron of $I$ (see Example~\ref{maximal-threshold} (2)
for definition). We also put ${\rm Vol}(P)$ for the volume of a
region $P$ in $\R^n$, with the Euclidean metric. Since the
deformation we consider is flat, it follows that
$\mathrm{in}_{\lambda}(\a)$ is also $\m$-primary and
$$\ell_R(R/\a)=\ell_R(R/\mathrm{in}_{\lambda}(\a)) \ge \mathrm{Vol}\left(\R^d_{\ge 0} \smallsetminus P(\mathrm{in}_{\lambda}(\a))\right),$$
where the inequality follows from \cite[Lemma 1.3]{dFEM}.

On the other hand, by \cite[Prop. 5.3]{dF}, we have
$\tau(\mathrm{in}_{\lambda}(\a)^t) \subseteq
\mathrm{in}_{\lambda}(\tau(\a^t))$ for all $t >0$. This implies that
$\pt^J(\a) \ge
\pt^{\mathrm{in}_{\lambda}(J)}(\mathrm{in}_{\lambda}(\a))$. Note
also that since $J$ is generated by monomials, we have
$\mathrm{in}_{\lambda}(J)=J$. Thus, we can reduce to the case when
$\a$ is generated by  monomials in $x_1, \dots, x_d$. That is, it is
enough to show that for every $\m$-primary monomial ideal $\a
\subseteq R$,
$$\mathrm{Vol}\left(\R^d_{\ge 0} \smallsetminus P(\a)\right) \ge \frac{1}{d!}\left(\frac{d}{\pt^J(\a)} \right)^d a_1 \cdots a_d.$$

It follows from the description of $\pt^J(\a)$ in
Example~\ref{maximal-threshold} (2) that we have $(a_1,\ldots,a_d)
\in \partial(\pt^J(a)\cdot P(\a))$. We can find a hyperplane
$H_q:=u_1/b_{1}+ \dots +u_d/b_{d}=1$ passing through the point
$(a_1, \dots, a_d)$ such that
$$H^+:=\left\{(u_1, \dots, u_d) \in \R^d_{\ge 0} \mid \frac{u_1}{b_{1}}+ \dots +\frac{u_d}{b_{d}} \ge 1 \right\} \supseteq \pt^J(\a)\cdot P(\a).$$
Therefore, we have
$$\mathrm{Vol}\left(\R^d_{\ge 0}\smallsetminus P(\a)\right) \ge \mathrm{Vol}\left(\R^d_{\ge 0}\smallsetminus \frac{1}{\pt^J(\a)}H^+\right)
=\frac{b_{1} \dots b_{d}}{d!\cdot \pt^J(\a)^d}.$$ On the other hand,
since $H$ passes through $(a_1,\ldots,a_d)$, it follows that
 $a_1/b_{1}+ \cdots +a_d/b_{d}=1$.
Comparing the arithmetic and geometric means of $\{a_i/b_{i}\}_i$,
we see that
$$b_{1} \cdots b_{d}  \ge d^d \cdot a_1 \cdots a_d.$$
Thus, combining these two inequalities, we obtain that
$$\mathrm{Vol}\left(\R^d_{\ge 0}\smallsetminus P(\a)\right) \ge\frac{b_{1} \cdots b_{d}}{d!\cdot \pt^J(\a)^d}
\ge \frac{1}{d!}\left(\frac{d}{\pt^J(\a)}\right)^d a_1 \cdots a_d,
$$ as required.
\end{proof}

\begin{rem}
It might seem that in the above proof we have shown a stronger
assertion than the one in Conjecture~\ref{mult conj}, involving the
length instead of the multiplicity.
 However, the two assertions are equivalent: this follows from
 \cite[Corollary 3.8]{Mum} which says that for every zero-dimensional ideal $\a$ in a
 $d$-dimensional regular local ring $R$, we have
 $$\ell_R(R/\a)\geq\frac{e(\a)}{d!}.$$
\end{rem}

We can prove a graded version of Conjecture \ref{mult conj}. In
fact, we prove a more precise statement, which is valid
independently of the characteristic.

\begin{thm}\label{homogeneous}
Let $R=\bigoplus_{d \ge 0} R_d$ be an $n$-dimensional graded
Cohen-Macaulay ring with $R_0$ a field of arbitrary characteristic.
If $\a$ and $J$ are ideals generated by full homogeneous systems of
parameters for $R$, and if $\a^N\subseteq J$, then
$$e(\a) \ge \left(\frac{n}{n+N-1}\right)^n e(J).$$
\end{thm}

\begin{cor}\label{cor_homogeneous}
Let $R$ be as in the theorem, with ${\rm char}(R_0)=p>0$. If $\a$
and $J$ are ideals generated by full homogeneous systems of
parameters for $R$, then
$$e(\a)\ge\left(\frac{n}{\pt_{-}^J(\a)}\right)^n e(J).$$
\end{cor}

\begin{proof}
Note that each $J^{[q]}$ is again generated by a full homogeneous
systems of parameters. It follows from the theorem and from the
definition of $\nu^{J}_{\a}(q)$ that for every $q=p^e$ we have
$$e(\a)\geq\left(\frac{n}{n+\nu^J_{\a}(q)}\right)^ne(J^{[q]})=
\left(\frac{qn}{n+\nu^J_{\a}(q)}\right)^n e(J).$$ On the right-had
side we can take a subsequence converging to
$\left(\frac{n}{\pt_{-}^J(\a)}\right)^n e(J)$, hence we get the
inequality in the corollary.
\end{proof}

\begin{proof}[Proof of Theorem~\ref{homogeneous}]
Suppose that $\a$ is generated by a full homogeneous system of
parameters $x_1, \dots, x_n$ of degrees $a_1 \le \dots \le a_n$, and
that $J$ is generated by another homogeneous system of parameters
$f_1, \dots, f_n$ of degrees $d_1\leq\dots\leq d_n$. Define
nonnegative integers $t_1, \dots, t_{n-1}$ inductively as follows:
$t_1$ is the smallest integer $t$ such that $x_1^t \in J$. If $2\leq
i\leq n-1$, then $t_i$ is the smallest integer $t$ such that
$x_1^{t_1-1} \cdots x_{i-1}^{t_{i-1}-1}x_i^t \in J$. Note that we
have by assumption $N\geq t_1+\dots+t_{n-1}-n+1$.

We first show the following inequality for every $i=1, \dots, n-1$:
\begin{equation}\label{eq_homogeneous2}
t_1 a_1+\cdots+t_{i} a_i \geq d_1+\cdots+d_i.
\end{equation}
Let $I_i$ be the ideal of $R$ generated by $x_1^{t_1},
x_1^{t_1-1}x_2^{t_2}, \ldots, x_1^{t_1-1} \cdots
x_{i-1}^{t_{i-1}-1}x_i^{t_i}$. Note that the definition of the
integers $t_j$ implies that $I_i \subseteq J$. The natural
surjection of $R/I_i$ onto $R/J$ induces a comparison map between
their free resolutions (we resolve $R/J$ by the Koszul complex, and
$R/I_i$ by a Taylor-type complex). Note that the $i^{\rm th}$ step
in the Taylor complex for the monomials
$X_1^{t_1},X_1^{t_1-1}X_2^{t_2},\ldots,X_1^{t_1-1}\cdots
X_{i-1}^{t_{i-1}-1}X_i^{t_i}$ in a polynomial ring with variables
$X_1,\ldots,X_n$, is a free module of rank one, with a generator
corresponding to the monomial
$${\rm lcm}(X_1^{t_1},X_1^{t_1-1}X_2^{t_2},\ldots,X_1^{t_1-1}\cdots
X_{i-1}^{t_{i-1}}X_i^{t_i})=X_1^{t_1}\cdots
X_{i-1}^{t_{i-1}}X_i^{t_i}$$ (see \cite[Exercise 17.11]{eisenbud}).
It follows that the map between the $i^{\rm th}$ steps in the
resolutions of $R/I_i$ and $R/J$ is of the form
$$R(-t_1 a_1-\cdots-t_{i} a_i) \to \bigoplus_{1 \le v_1 < \dots < v_i \le n}R(-d_{v_1}-\dots-d_{v_i}).$$
In particular, unless this map is zero, we have
$$t_1
a_1+\cdots+t_{i} a_i\geq\min_{1 \le v_1 < \dots < v_i \le
n}(d_{v_1}+\cdots+d_{v_i})=d_1+\cdots+d_i.$$ We now show that this
map cannot be zero. If it is zero, then also the induced map
\begin{equation}\label{eq_homogeneous1}
\Tor_i^R(R/I_i, R/\b_i) \to \Tor_i^R(R/J, R/\b_i)
\end{equation}
 is zero, where $\b_i$ is the ideal generated by $x_1, \dots, x_i$. On the other
hand, using the Koszul complex on $x_1, \dots, x_i$ to compute the
above $\Tor$ modules, we see that the map (\ref{eq_homogeneous1})
can be identified with the natural map
$$(I_i\colon \b_i)/I_i \to (J\colon \b_i)/J.$$
Since $x_1^{t_1-1}\cdots x_i^{t_i-1}\in (I_i\colon \b_i)$, it
follows that
 $x_1^{t_1-1} \cdots x_i^{t_i-1}$
lies in $J$, a contradiction. This proves (\ref{eq_homogeneous2}).

We next prove the following inequality:
\begin{equation}\label{eq_homogeneous3}
t_1 a_1+\cdots+t_{n-1} a_{n-1}+(N-t_1-\dots-t_{n-1}+n-1)a_n \ge
d_1+\cdots+d_n.
\end{equation}
Since $\a^{N} \subseteq J$, we have
\begin{equation}\label{eq_homogeneous4}
(x_1^{N}, \dots, x_n^{N})\colon J \subseteq (x_1^{N}, \dots,
x_n^{N})\colon \a^{N}=(x_1^{N}, \dots, x_n^{N})+\a^{(n-1)(N-1)}.
\end{equation}
 On the other hand, the ideal $(x_1^{N}, \dots,
x_n^{N})\colon J$ can be described as follows. If we write
$x_i^N=\sum_{j=1}^nb_{ij}f_j$, then using the Koszul resolutions of
$R/J$ and $R/(x_1^N,\ldots,x_n^N)$ one sees that multiplication by
$D={\rm det}(b_{ij})$ gives an injection $R/J\hookrightarrow
R/(x_1^N,\ldots,x_n^N)$, hence $J=(x_1^N,\ldots,x_n^N)\colon D$.
Moreover, we also get
$$(x_1^N,\ldots,x_n^N)\colon
J=(x_1^N,\ldots,x_n^N,D)$$ (see, for example, \cite[Prop. 2.6]{PS};
note that the statement therein requires $R$ to be regular, but this
condition is not used). It follows from the above description that
$D$ is homogeneous, and $\deg(D)= N(a_1+\dots+a_n)-(d_1+\dots+d_n)$.

It follows from (\ref{eq_homogeneous4}) that after possibly adding
to $D$ an element in $(x_1^N,\ldots,x_n^N)$, we may write
$$D=\sum_{m_1+\dots+m_n=(n-1)(N-1)} c_{m_1 \dots m_n} x_1^{m_1} \dots x_{n}^{m_{n}},$$
where all $c_{m_1,\ldots,m_n}$ are homogeneous. Since $x_1^{t_1-1}
\cdots x_{n-1}^{t_{n-1}-1}$ is not in $J=(x_1^N,\ldots,x_n^N)\colon
D$, we see that
$$D\not\in (x_1^N,\ldots,x_n^N)\colon x_1^{t_1-1}\cdots
x_{n-1}^{t_{n-1}-1}= (x_1^{N-t_1+1}, \dots, x_{n-1}^{N-t_{n-1}+1},
x_n^{N}).$$
 Thus there is some $(m_1,\ldots,m_n)$ with $\sum_jm_j=(n-1)(N-1)$
 and $m_j\leq N-t_j$ for all $j\leq n-1$, such that $c_{m_1\ldots
 m_m}\neq 0$. We deduce that the degree of $D$ is at least as large
as the smallest degree of such a monomial $x_1^{m_1}\cdots
x_n^{m_n}$,
 hence
\begin{align*}
\deg D&=N(a_1+\dots+a_n)-(d_1+\dots+d_n)\\
&\ge
(N-t_{1})a_1+\dots+(N-t_{n-1})a_{n-1}+(t_1+\dots+t_{n-1}-n+1)a_n,
\end{align*}
which implies the inequality (\ref{eq_homogeneous3}).

To finish the proof, we will use the following claim.
\begin{cl}
Let $\alpha_i,\beta_i,\gamma_i$ be real numbers, for $1\leq i\leq
n$. If $1=\gamma_1\leq\gamma_2\leq\ldots\leq\gamma_n$, and if
$\gamma_1\alpha_1+\dots+\gamma_i\alpha_i \ge
\gamma_1\beta_1+\dots+\gamma_i\beta_i$ for all $i=1, \dots, n$, then
$\alpha_1+\dots+\alpha_n \ge \beta_1+\dots+\beta_n$.
\end{cl}
\begin{proof}[Proof of Claim]
Let $\lambda_i=\alpha_i-\beta_i$ for $1 \le i \le n$, so that
$\gamma_1\lambda_1+\dots+\gamma_i\lambda_i \ge 0$ for all $i=1,
\dots, n$. We prove that $\lambda_1+\dots+\lambda_n \ge 0$ by
induction on $n$, the case $n=1$ being trivial. Suppose that $n>1$
and that there is $i$ such that $\lambda_i<0$ (otherwise the
assertion to prove is clear). We must have $i \ge 2$,  and since
$\gamma_i \ge \gamma_{i-1}$, it follows that $\gamma_i \lambda_i \le
\gamma_{i-1}\lambda_i$. Let us put $\gamma_j'=\gamma_j$ for $1 \le j
\le i-1$ and $\gamma_j'=\gamma_{j+1}$ for $i \le j \le n-1$. Define
also $\lambda_j'=\lambda_j$ for $1 \le j \le i-2$,
$\lambda_{i-1}'=\lambda_{i-1}+\lambda_i$ and
$\lambda_j'=\lambda_{j+1}$ for $i \le j \le n-1$. It is
straightforward to check that
$\gamma_1'\lambda_1'+\dots+\gamma_{j}'\lambda_{j}' \ge 0$ for all
$j=1, \dots, n-1$, hence the induction hypothesis implies
$\lambda_1+\dots+\lambda_n=\lambda_1'+\dots+\lambda_{n-1}'\ge 0$.
\end{proof}

We now
 set $\alpha_i=t_i$ for $1 \le i \le n-1$ and
$\alpha_n=N-t_1-\dots-t_{n-1}+n-1$. We put $\beta_i=d_i/a_i$ and
$\gamma_i=a_i/a_1$ for $1 \le i \le n$. Since $a_1\leq\dots\leq
a_n$, we deduce $1=\gamma_1\leq\dots\leq\gamma_n$. Moreover, using
(\ref{eq_homogeneous2}) and (\ref{eq_homogeneous3}), we get
$\gamma_1\alpha_1+\dots+\gamma_i\alpha_i \ge
\gamma_1\beta_1+\dots+\gamma_i\beta_i$ for $1 \le i \le n$. Using
the above claim, we conclude that
$${N+n-1}=\alpha_1+\dots+\alpha_n \ge \beta_1+\dots +\beta_n=\left(\frac{d_1}{a_1}+\dots+\frac{d_n}{a_n}\right).$$
Comparing the arithmetic and geometric means of $\{d_i/a_i\}_i$, we
see that
$$(N+n-1)^n a_1 \dots a_n \ge n^n d_1 \dots d_n.$$
Since $e(\a)=a_1\cdots a_n$ and $e(J)=d_1\cdots d_n$, this concludes
the proof.
\end{proof}

When $J$ is not necessarily a parameter ideal, we can prove another
inequality involving the F-threshold $\pt^J(\a)$, generalizing the
results in \cite{dFEM} and \cite{TW}.
\begin{prop}\label{another}
If $(R,\m)$ is a $d$-dimensional regular local ring of
characteristic $p>0$, and if $\a, J$ are $\m$-primary ideals in $R$,
then we have the following inequality:
$$e(\a) \ge \left(\frac{d}{\pt^J(\a)}\right)^d (\pt^J(\m)-d+1).$$
\end{prop}
\begin{proof}
As in the proof of Theorem~\ref{diagonal}, we do a reduction to the
monomial case. We first see that it is enough to show that if
 $R$ is the polynomial ring $k[x_1,\ldots,x_d]$ and $\m=(x_1,\ldots,x_d)$, and $\a$, $J$
are $\m$-primary ideals, then
\begin{equation}\label{eq_reduce}
\ell(R/\a)\geq \frac{1}{d!} \left(\frac{d}{\pt^J(\a)}\right)^d
(\pt^J(\m)-d+1).
\end{equation}

\begin{cl}
We can find monomial ideals $\a_1$ and $J_1$
 such that
\begin{equation}\label{eq_reduce2}
 \ell_R(R/\a)=\ell_R(R/\a_1),\,\, \pt^{J}(\a)\geq
 \pt^{J_1}(\a_1),\,\,{\rm and}\,\,\pt^J(\m)=\pt^{J_1}(\m).
\end{equation}
\end{cl}

This reduces the proof of (\ref{eq_reduce})  to the case when both
$\a$ and $J$ are monomial ideals.

\begin{proof}[Proof of claim]
We do a two-step deformation to monomial ideals. We consider first a
flat deformation of $\a$ and $J$ to $\a'$ and $J'$, respectively,
where for an ideal $I\subseteq R$, we denote by $I'$ the ideal
defining the respective tangent cone at the origin. We then fix a
monomial order $\lambda$, and consider a Gr\"{o}bner deformation of
$\a'$ and $J'$ to $\a_1:={\rm in}_{\lambda}(\a')$ and $J_1:= {\rm
in}_{\lambda}(J')$, respectively. It follows as in the proof of
Theorem~\ref{diagonal} that the first two conditions in
(\ref{eq_reduce2}) are satisfied. For the third condition, in light
of Example~\ref{maximal-threshold} (3) it is enough to show that
$$\m^r\subseteq J\,\,{\rm iff}\,\,\m^r\subseteq {\rm
in}_{\lambda}(J).$$

It is clear that if $\m^r\subseteq J$, then $\m^r\subseteq J'$ and
$\m^r\subseteq J_1$. For the converse, suppose that $\m^r\subseteq
J_1$. Since $J'$ and $J_1$ are both homogeneous ideals, and since
$\dim_k(R/J_1)_r=\dim_k(R/J')_r$ (see \cite[Ch. 15]{eisenbud}), it
follows that $\m^r\subseteq J'$ (note that if $I$ is a homogeneous
ideal in $R$, then $\m^r\subseteq I$ if and only if $(R/I)_r=0$). We
know that $\m^s\subseteq J$ for some $s$, hence in order to prove
that $\m^r\subseteq J$ it is enough to show the following assertion:
if $\m^t\subseteq J'$ and $\m^{t+1}\subseteq J$, then $\m^t\subseteq
J$. It is easy to check that $(J\cap\m^t)'=J'\cap\m^t$, and since
$\m^{t+1}\subseteq J$, we see that $J\cap\m^t$ is homogeneous, hence
$$\m^t\subseteq J'\cap\m^t=(J\cap\m^t)'=J\cap\m^t.$$
\end{proof}

We return to the proof of Proposition~\ref{another}. From now on we
assume that $\a$ and $J$ are $\m$-primary monomial ideals. Arguing
as in the proof of Theorem~\ref{diagonal}, and using
Example~\ref{maximal-threshold} (3), we see that it is enough to
show
$$\mathrm{Vol}\left(\R^d_{\ge 0} \smallsetminus P(\a)\right) \ge \frac{1}{d!}\left(\frac{d}{\pt^J(\a)}\right)^d (r+1),$$
where $r:=\max\{s \in \Z_{\ge 0} \mid \m^s \not\subseteq J \}$. By
definition, we can choose a monomial $x_1^{r_1} \cdots x_d^{r_d}$ of
degree $r$ that is not contained in $J$. Since $\tau(\a^{\pt^J(\a)})
\subseteq J$ by Proposition \ref{regular}, this monomial cannot
belong to $\tau(\a^{\pt^J(\a)})$. Using the description of
generalized test ideals of monomial ideals (see
 \cite[Theorem 4.8]{HY}), this translates as
 $$(r_1+1,\ldots,r_d+1)\not\in {\rm Int}(\pt^J(\a)\cdot P(\a)).$$
Therefore we can find a hyperplane
$H:u_1/a_1+\dots+u_d/a_d=\pt^J(\a)$ passing through the point
$(r_1+1, \dots, r_d+1)$ such that
\begin{equation}\label{inclusion}
H^+:=\left\{(u_1, \dots, u_d) \in \R^d_{\ge 0} \mid
\frac{u_1}{a_{1}}+ \dots +\frac{u_d}{a_{d}} \ge\pt^J(\a) \right\}
\supseteq \pt^J(\a) \cdot P(\a).
\end{equation}
Note that we get
$\pt^J(\a)=(1+r_1)/a_1+\dots+(1+r_d)/a_d$. Comparing the arithmetic
and geometric means of $\{(1+r_i)/a_i\}_i$, we see that
$$\left(\frac{\pt^J(\a)}{d}\right)^d= \left( \frac{1+r_1}{da_1} + \dots + \frac{1+r_d}{da_d} \right)^d \ge
 \frac{(1+r_1) \dots (1+r_d)}{a_1 \dots a_d} \ge  \frac{1+r}{a_1 \dots  a_d}. $$
On the other hand, (\ref{inclusion}) implies
\begin{align*}
\mathrm{Vol}\left(\R^d_{\ge 0}\smallsetminus P(\a)\right) &\ge \mathrm{Vol}\left(\R^d_{\ge 0}\smallsetminus (1/{\pt^J(\a)})H^+\right)\\
&=\frac{a_1 \dots a_d}{d!}\\
&\ge \frac{1}{d!}\left(\frac{d}{\pt^J(\a)}\right)^d\left(r+1\right).
\end{align*}
\end{proof}

\bigskip

\begin{acknowledgement}
We thank all the participants in the AIM workshop ``Integral
Closure, Multiplier Ideals and Cores" for valuable comments.
Especially, we are indebted to Mel Hochster for the suggestion to
introduce the invariant $\widetilde{\pt}^J(\a)$. We are also
grateful to Tommaso de Fernex for suggesting Theorem~\ref{diagonal}.
The first author was partially supported by NSF grant DMS-0244405.
The second author was partially supported by NSF grant DMS-0500127
and by a Packard Fellowship. The third and fourth authors were
partially supported by Grant-in-Aid for Scientific Research 17740021
and 17540043, respectively. The third author was also partially
supported by Special Coordination Funds for Promoting Science and
Technology from the Ministry of Education, Culture, Sports, Science
and Technology of Japan.
\end{acknowledgement}

\end{document}